\documentclass[11pt,reqno]{amsart}

\usepackage{pgf,tikz}
\usepackage{mathrsfs}
\usepackage{geometry, amsmath, amssymb, enumerate, bbm, float}

\usepackage{pgfplots}
\usepackage{subfig}
\pgfplotsset{compat=newest}
\usepgfplotslibrary{fillbetween} %per regioni planari

\usetikzlibrary{shapes.geometric,arrows,fit,matrix,positioning,trees,snakes,patterns}
\tikzset
{
    treenode/.style = {circle, draw=black, align=center, minimum size=1cm},
    subtree/.style  = {isosceles triangle, draw=black, align=center, minimum height=0.5cm, minimum width=1cm, shape border rotate=90, anchor=north}
}

\usepackage{marginnote, amssymb, mathrsfs, latexsym, verbatim, pdfsync, color, mathabx}
\usepackage[utf8]{inputenc}
\usepackage[colorlinks=true, linkcolor=blue]{hyperref}
\usepackage{mathtools}

\usepackage[colorlinks=true, linkcolor=blue]{hyperref}
\usepackage{mathtools}

\def\Re{\operatorname{Re}}%{\, \text{\rm Re\,}}
\numberwithin{equation}{section}

\DeclareFontFamily{U}{mathx}{\hyphenchar\font45}
\DeclareFontShape{U}{mathx}{m}{n}{
      <5> <6> <7> <8> <9> <10>
      <10.95> <12> <14.4> <17.28> <20.74> <24.88>
      mathx10
      }{}
\DeclareSymbolFont{mathx}{U}{mathx}{m}{n}
\DeclareFontSubstitution{U}{mathx}{m}{n}
\DeclareMathAccent{\widecheck}{0}{mathx}{"71}
\DeclareMathAccent{\wideparen}{0}{mathx}{"75}
\makeatletter
\newcommand{\leqnomode}{\tagsleft@true}
\newcommand{\reqnomode}{\tagsleft@false}
\makeatother
\theoremstyle{plain}

\newtheorem{thm}{Theorem}[section]

\newtheorem{cor}[thm]{Corollary}
\newtheorem{lem}[thm]{Lemma}
\newtheorem{prop}[thm]{Proposition}

\theoremstyle{definition}
 
\newtheorem{oss}[thm]{Remark}

\theoremstyle{remark}

\begin{document}

\title[Hardy--Littlewood fractional maximal operators]{Hardy--Littlewood fractional maximal operators on homogeneous trees} 

 \author[M.\ Levi]{Matteo Levi}
\address{ MaLGa Center - DIBRIS  - Università di Genova, Genoa, Italy.}
\email{m.l.matteolevi@gmail.com}

% \author[S. \ Meda]{Stefano Meda}
%\address{Dipartimento di Matematica e Applicazioni,
%Universit\`a degli Studi di Milano-Bicocca, Via Cozzi 53, 20125 Milano, Italy}
%\email{stefano.meda@unimib.it}

 \author[F.\ Santagati]{Federico Santagati}
\address{Dipartimento di Scienze Matematiche ``Giuseppe Luigi Lagrange'',
  Politecnico di Torino, Corso Duca degli Abruzzi 24, 10129 Torino,
  Italy - Dipartimento di Eccellenza 2018-2022}
\email{federico.santagati@polito.it}

%\author[M.\ Vallarino]{Maria Vallarino}
%\address{Dipartimento di Scienze Matematiche ``Giuseppe Luigi Lagrange'', Politecnico di Torino, Corso Duca degli Abruzzi 24, 10129 Torino, Italy - Dipartimento di Eccellenza 2018-2022}
%\email{maria.vallarino@polito.it}


\thanks{{\em Math Subject Classification} 05C05, 42B25, 43A85}

%\thanks{Levi, Santagati and Vallarino  are members of the Gruppo Nazionale per l'Analisi
%  Matematica, la Probabilit\`a e le loro Applicazioni (GNAMPA) of the
%  Istituto Nazionale di Alta Matematica (INdAM) and of the project  ``Harmonic analysis on continuous and discrete
%structures'' funded by Compagnia di San Paolo (Cup E13C21000270007) }  

\thanks{The authors are members of the Gruppo Nazionale per l'Analisi
  Matematica, la Probabilit\`a e le loro Applicazioni (GNAMPA) of the
  Istituto Nazionale di Alta Matematica (INdAM) and of the project  ``Harmonic analysis on continuous and discrete
structures'' funded by Compagnia di San Paolo (Cup E13C21000270007) }

\begin{abstract} %We study the range of exponents $(p,q)$ for which the Hardy--Littlewood fractional maximal operator on homogeneous trees is of strong, weak or restricted weak type $(p,q)$ and we study the optimality of such a range.
 We study the mapping properties of the Hardy--Littlewood fractional maximal operator between Lorentz spaces of the homogeneous tree and discuss the optimality of all the results.
\end{abstract}

\maketitle
\section{Introduction}
In the generality of any metric measure space $(X,\mu)$ one can define the (centered) \textit{fractional maximal operator} of parameter $\gamma>0$ as the operator $\mathcal M^{\gamma}$ acting on locally integrable functions $f\in \mathbb{C}^X$ as follows,
\begin{equation}\label{max}
\mathcal M^{\gamma}f(x)=\sup_{r>0} \frac{1}{\mu(B_r(x))^{\gamma}}\int_{B_r(x)} |f|d\mu,\qquad  x\in X. 
\end{equation}
We shall simply write $\mathcal{M}$ for the classical Hardy--Littlewood maximal operator, which corresponds to the choice of parameter $\gamma=1$.

 If the measure $\mu$ is doubling, the space $X$ is of homogeneous type in the sense of Coifman and Weiss and $\mathcal{M}$ is bounded on $L^p(\mu)$ for every $p>1$ and is of weak type $(1,1)$  \cite{CW}. This ``maximal theorem" has a fundamental role in harmonic analysis, in particular in the theory of singular integrals. Many attempts have been made to extend this theory beyond the setting of spaces of homogeneous type.
In \cite{NTV97} Nazarov, Treil, and Volberg, under the hypothesis that the measure $\mu$ has at most polynomial growth, managed to prove that a \textit{modified} maximal operator (obtained substituting $\mu(B_r(x))$ with $\mu(B_{3r}(x))$ in \eqref{max}) is of weak type (1,1). However, even in the polynomial growth regime, the same is not true in general for $\mathcal{M}$ \cite{A}.
In \cite{CMM09,CMM10} Carbonaro, Mauceri and Meda developed a theory of singular integrals for metric measure spaces satisfying the isoperimetric and the so-called middle point property. In their setting, the maximal theorem holds for the dyadic maximal operator, but there is no equivalent result available for $\mathcal{M}$. Despite these difficulties, there are some interesting classes of metric measure spaces where the theory seems to behave well even in the absence of the doubling property. This is the case for non-compact symmetric spaces, where $\mathcal M$ is bounded on $L^p$ for every $p>1$ \cite{clerc} and is of weak type $(1,1)$ \cite{S}, and for homogeneous trees, where the maximal theorem for $\mathcal{M}$ was proved independently by Cowling, Meda, and Setti \cite{CMS10} and by Naor and Tao \cite{NT}. As observed in \cite{CMS10}, the result on the homogeneous tree can be also deduced from an older theorem by Rochberg and Taibleson \cite{rochberg}. In these settings it is also known that the fractional maximal operator $\mathcal{M}^{1/2}$ is of restricted weak type $(2,2)$; this was proved by Ionescu on non-compact symmetric spaces of real rank one \cite{Ionescu1} and by Veca on homogeneous trees \cite{V}.

One aspect where the nondoubling theory diverges from the classical one is the following. Taking the supremum in \eqref{max} over \textit{all} the balls containing $x$, and not only on those centered at it, one obtains the so-called uncentered (fractional) maximal operator, which we denote by $\widetilde{\mathcal M}^{\gamma}$. While in spaces of homogeneous type $\mathcal{M}^{\gamma}$ and $\widetilde{\mathcal M}^{\gamma}$ are comparable, in the nondoubling setting they may be very distinct objects. For instance, ${\mathcal M}$ is of weak type $(1,1)$ with respect to any Borel measure in $\mathbb{R}^n$, $n\geq 1$ (see for instance \cite{Guzman}), but the uncentered operator $\widetilde{\mathcal M}$ associated to the gaussian measure is not of weak type $(1,1)$ on $\mathbb{R}^2$ \cite{SJ}. On the other hand, in some situations boundedness results for $\mathcal M^\gamma$ can be transferred to $\widetilde{\mathcal M}^{2\gamma}$. For example, if $\mu$ is a nondoubling measure such that $\mu(B_r(x))\approx k^r$ for some constant $k$ and every $x\in X$, if $x$ belongs to a ball $B$ of radius $r$ then $B\subseteq B_{2r}(x)$, from which follows
\begin{equation}\label{centered to uncentered}
    \frac{1}{\mu(B)^{\gamma}}\int_{B} |f|d\mu\le \frac{1}{\mu(B)^{\gamma}}\int_{B_{2r}(x)} |f|d\mu\approx \frac{1}{\mu(B_{2r}(x))^{\gamma/2}}\int_{B_{2r}(x)}|f|d\mu.
\end{equation}
This is the case, again, on symmetric spaces and homogeneous trees, and hence, the results by Ionescu and Veca on $\mathcal{M}^{1/2}$ also give that the uncentered maximal operator $\widetilde{\mathcal M}$ is of weak type $(1,1)$ in the respective settings.

That said, the interest in fractional maximal operators cannot be reduced to the study of uncentered maximal operators in nondoubling settings. Weighted $L^p$ norm inequalities for fractional maximal operators on $\mathbb{R}^n$ are a classical object of study \cite{muckenhoupt, welland, adams} and are intimately related to boundedness properties of Riesz potentials. More recently, also the mapping properties of $\mathcal{M^\gamma}$ on Sobolev, Hölder, Campanato, Morrey and variable exponent $L^p$ spaces have been studied, in $\mathbb{R}^n$ and in metric spaces \cite{saksman, heikkinen2013fractional,heikkinen2013mapping,capone}. Also in the discrete setting mapping properties of maximal operators are an active area of research. Weak type estimates for $\mathcal{M}$ have been obtained on $\mathbb{Z}^d$ \cite{carneiro2012endpoint} and on graphs fulfilling some geometric properties \cite{soria}. On the homogeneous tree also a weighted theory for strong and weak type estimates is developing, both for $\mathcal{M}$ \cite{ombrosiLp,ombrosi2021fefferman} and for $\mathcal{M}^\gamma$ with $\gamma\in (0,1)$ \cite{GR}.

The aim of this note is to complement the works \cite{CMS10, NT} and \cite{V} extending the study of the mapping properties of $\mathcal{M}^\gamma$ on Lebesgue and Lorentz spaces of the homogeneous tree $T$ also to values of $\gamma\neq 1,1/2$. Here, as in the aforementioned works, $T$ is endowed with the standard graph distance and the counting measure $|\cdot|$. Observe that if $T$ is homogeneous of order $k$, i.e., $|B_1(x)|=k+2$ for every $x\in T$, then $|B_r(x)|\approx k^r$, for every $x\in T, r\in \mathbb{N}$. Hence, one should keep in mind that all the positive boundedness results we obtain for $\mathcal{M}^\gamma$ transfer to boundedness results for the uncentered fractional operator $\widetilde{\mathcal M}^{2\gamma}$ by means of \eqref{centered to uncentered}. The paper is organized as follows.
In Section \ref{sec: boundedness} we prove endpoint results for $\mathcal{M}^\gamma$ analogous to the aforementioned ones for values of $\gamma\neq 1,1/2$. In particular, in Theorems \ref{complex interpolation} and \ref{p: homtree} we show that if $\gamma\in (1/2,1)$, then $\mathcal{M}^\gamma$ is of restricted weak-type $(1/\gamma,1/\gamma)$ and is bounded from $L^{1/(1-\gamma), 1/[2(1-\gamma)]}$ to $L^{1/(1-\gamma), \infty}$, and if $\gamma\in (0,1/2]$, then it is of restricted weak-type $(1/(1-\gamma),1/\gamma)$. While Theorem \ref{p: homtree} follows from a rather simple application of a sharpened version of the Kunze-Stein phenomenon for Lorentz spaces on the homogeneous tree \cite{CMS98}, for Theorem \ref{complex interpolation} a different approach based on a complex interpolation argument is needed. We also provide strong type results. In Theorem \ref{strong large gamma} we prove that, for any $\gamma\in(0,1]$,  $\mathcal{M^\gamma}$ maps $L^p$ continuously to $L^q$ when $1\leq p\leq q\leq\infty$ and (i) $q>1/\gamma$ and $p<1/(1-\gamma)$ or (ii) $p=1/(1-\gamma), q=\infty$.

In Section \ref{sec:opt} we discuss the optimality of the results of Section \ref{sec: boundedness}. In Theorem \ref{counterexample 2} we prove that Theorems \ref{complex interpolation} and \ref{p: homtree} are optimal in the sense that, if $t \in [1,\infty)$, then $\mathcal{M}^{\gamma}$ is unbounded from $L^{1/(1-\gamma),s}$ to $L^{1/(1-\gamma),t}$ and from $L^{p,s}$ to $L^{1/\gamma,t}$, for every $p \in [1,\infty)$ and $s\in[1,\infty]$. Proposition \ref{optaxis} is a partial converse to Theorem \ref{strong large gamma}: it states that $\mathcal{M^\gamma}$ does not map $L^p$ continuously to $L^q$ for all the values of $p$ and $q$ not satisfying (i) and (ii) and not lying on the open segment $(1-\gamma,1/q)$ with $0<1/q<\min\{\gamma, 1-\gamma\}$; to prove or disprove the $L^p$ to $L^q$ boundedness for points on such a critical segment seems to be a difficult problem, which we leave open for the time being. Theorem \ref{optM12} shows that $\mathcal{M}^{1/2}$ is unbounded from $L^{2,s}$ to $L^{2,\infty}$ for every $s>1$. This is an optimality result for the aforementioned  theorem by Veca and should be considered as the tree analogue of \cite[Section 4]{Ionescu2}. If Theorem \ref{optM12} extends to values of $\gamma\neq 1/2$ remains an open problem. We discuss it and provide evidence that if such an extension is possible it is far from being straightforward, see Proposition \ref{prop:radial}.

We end the paper by comparing our strong type results with those in \cite{GR}. In particular, we show that the sufficient condition for boundedness given in \cite{GR} is not strong enough to provide a positive answer to our open question on strong boundedness on the critical segment.

%some partial answers in the note.

%We also discuss strong type results, providing values of $p,q$ such that $\mathcal{M^\gamma}$ maps $L^p$ continuously to $L^q$ (Proposition \ref{holder}, Theorem \ref{strong large gamma}), and values for which it does not (Propositions \ref{optaxis} and \ref{counterexample 2}). We can determine weather the $(p,q)$ strong boundedness holds or not for all values $(1/p,1/q)\in[0,1]^2$ except for those lying on the segment the open segment $(1-\gamma,1/q)$ with $0<1/q<\min\{\gamma, 1-\gamma\}$: to prove or disprove the $L^p$ to $L^q$ boundedness for points on such a segment seems to be a difficult problem, which we leave open for the time being.

%{\color{blue} We remark that, since $|B_r(x)|\approx k^r$, for every $x\in T, r\in \mathbb{N}$, all the positive (restricted weak, weak or strong) boundedness results proved for $\mathcal{M}^\gamma$ transfer to boundedness results for the uncentered fractional operator $\widetilde{\mathcal M}^{2\gamma}$ by means of \eqref{centered to uncentered}.
%}
It is natural to ask whether analogous results may be proved on nonhomogeneous trees. Some of them extend quite naturally to certain classes of trees, for some others, the proofs are very specific to the homogeneous case and genuinely new approaches seem to be needed. This is work in progress.

In the previous pages as well as in those to come, we adopt the convention of writing $A\lesssim B$ if there exists a positive constant $c$, not depending on variables but possibly depending on parameters (which are the variables and which the parameters should be clear from time to time by the context) such that $A\leq c B$, and $A\approx B$ if it is both $A\lesssim B$ and $B\lesssim A$.

%Kunze-Stein Phenomenon: proved in [ R. A. KUNZE and E. M. STEIN, Uniformly bounded representations and harmonic analysis of the 2 x 2 unimodular group] for $SL(2,\mathbb{R})$, extended by cowling to connected semisimple Lie groups with finite center [M. COWLING, The Kunze-Stein phenomenon], and extended to Lorentz spaces by Cowling, Meda, Setti [, Herz's "principe de majoration" and the Kunze-Stein phenomenon, in Harmonic Analysis and Number Theory,   M. Cowling, S. Meda and A.G. Setti, 'An overview of harmonic analysis on the group of isometries of a homogeneous tree', Exposition. Math.] and the endpoint by Ionescu in [An Endpoint Estimate for the Kunze-Stein Phenomenon and Related Maximal Operators] for rank 1 non compact c.s.s. Lie groups. Later adapted to homogeneous tree by [VECA].

\section{Preliminaries}\label{sec:preliminaries}

Let $T$ be a homogeneous tree, i.e., a connected graph with no cycles where each vertex  has exactly  $k+1$ neighbors for some $k \ge 2$. Nothing will ever depend on the specific value of $k$, which we assume fixed once for all. We identify $T$ with its set of vertices and endow it with the standard graph distance $d$, counting the number of edges along the shortest path connecting two vertices. We fix an (arbitrary) distinguished point $o\in X$ and we abbreviate $d(o,x)$ with $\Vert x\Vert$. For every $x\in T$ and $r\in\mathbb N$ we denote by $B_r(x)$ the ball centered at $x$ of radius $r$, i.e., $B_r(x)=\{y\in T: d(y,x)\leq r\}$ and by $S_r(x)$ the sphere centered at $x$ of radius $r$, i.e., $S_r(x)=\{y\in T: d(y,x)= r\}$.
We endow $T$ with the counting measure and for every subset $E$ of $T$ we denote by $|E|$ its cardinality. Observe that $|B_r(x)|=|B_r(o)|\approx k^r$ for every $x\in T$ and $r\in \mathbb{N}$. It follows that the fraction maximal operator of parameter $\gamma>0$ \eqref{max} on $(T,|\cdot|)$ takes the form
$$
\mathcal M^{\gamma}f(x)=\sup_{r\in\mathbb N} \frac{1}{|B_r(x)|^{\gamma}}\sum_{y\in B_r(x)} |f(y)|=\sup_{r\in\mathbb N} \frac{1}{|B_r(o)|^{\gamma}}\sum_{y\in B_r(x)} |f(y)|,\qquad  f\in \mathbb{C}^T, x\in T.
$$
We are studying the mapping properties of $\mathcal{M}^\gamma$ between Lebesgue and Lorentz spaces on $T$. For every $p\in [1,\infty)$, we denote by $L^p$ the space of functions $f \in \mathbb C^T$ such that $\|f\|_{p}^p=\sum_{x\in T}|f(x)|^p<\infty$ and by $L^{\infty}$  the space of functions $f \in \mathbb C^T$ such that $\|f\|_{{\infty}}=\sup_{x\in T}|f(x)|<\infty$. For every $p\in[1,\infty]$, we denote by $p'$ its conjugate exponent, i.e., $1/p+1/p'=1$. We also introduce the Lorentz spaces $L^{p,s}$ and $L^{p,\infty}$ on $T$, which for $p\in [1,\infty)$ and $s \in [1,\infty)$ are defined by
$$
L^{p,s}=\Big\{f \in \mathbb C^T: \,\|f\|_{{p,s}}=\bigg(p\int_0^{+\infty}\lambda^s|\{x\in T: |f(x)|> \lambda\}|^\frac{s}{p} \frac{d\lambda}{\lambda}\bigg)^{1/s}<\infty\Big\},
$$
and 
$$
L^{p,\infty}=\{f \in \mathbb C^T: \,\|f\|_{{p,\infty}}=\sup_{\lambda>0} \lambda |\{x\in T: |f(x)|> \lambda\}|^{1/p} <\infty\}.
$$
By convention, we set $L^{\infty,\infty}=L^\infty$.  For any $p,q\in[1,\infty]$, we say that an operator is of strong (weak) type $(p,q)$ if it is bounded from $L^p$ to $L^q$ ($L^{q,\infty}$ respectively), and for any $p,q\in[1,\infty)$ we say that an operator is of \emph{restricted} weak type $(p,q)$ if it satisfies the weak
type $(p,q)$ condition when it is restricted to characteristic functions of
sets of finite measure. When $q>1$, this is equivalent to the boundedness from $L^{p,1}$ to $L^{q,\infty}$ (see \cite[Th. 3.13]{SW}).

Let us recall that discrete Lorentz spaces enjoy the following embedding property, which will be of good use in the next section.

\begin{lem}\label{lem:inclusions}
If  an operator is of restricted weak type $(p_0,q_0)$ for some $p_0,q_0 \in [1,\infty)$, then it is of restricted weak type (strong type) $(p,q)$ for every $1 \le p\le p_0$ and $q_0\le q\le \infty$ (respectively, for every $1 \le p< p_0$ and $q_0<q\le \infty$).  
\end{lem}
\begin{proof}
For what concerns the restricted weak type boundedness it suffices to recall that for any $p,p_0 \in [1,\infty)$ and $s\in[1,\infty]$, the continuous inclusion $L^{p,s} \hookrightarrow L^{p_0,s}$ holds if $p \le p_0$. The statement regarding the strong type boundedness then follows from the general Marcinkiewicz interpolation theorem  \cite[Theorem 5.3.2]{BL}.
\end{proof}

Another important auxiliary result is the following formula for the Lorentz space norm of a radial function on $T$, which follows from a result of Pytlik \cite{P} (see also \cite[Lemma A3]{CMS98}). Here, with harmless abuse of notation, we denote by $f(n)$ the value that a radial function $f$ takes on $S_n(o)$.

\begin{lem}\label{lem: P}
Let $f$ be a radial function on $T$. Then, for every $p \in [1,\infty)$ and $s\in [1,\infty]$,
$$
\Vert f\Vert_{p,s}\approx\Vert g\Vert_{L^s(\mathbb{N)}},
$$
where $g(n)=f(n)k^{n/p}.$
\end{lem}

When studying the mapping properties of $\mathcal{M}^\gamma$, two simple but crucial observations are in order.
\begin{oss}\label{remark 1}
For every $\gamma>0$, $\mathcal{M}^\gamma$ is unbounded from $L^p$ to $L^q$ when $p>q$, since the identity is not bounded on the same spaces and $\mathcal{M}^\gamma f \geq |f|$ pointwise, for every $\gamma>0$.
\end{oss}

\begin{oss}\label{remark 2}
For every $\gamma>0$, $r \in \mathbb N$, let $a_{r,\gamma}$ denote the radial function $|B_r(o)|^{-\gamma}{\chi}_{B_r(o)}$. Then, for every $f\in \mathbb{C}^T$ and every $r\in \mathbb N$ and  $x$ in $T$,
\begin{align*}
 \frac{1}{|B_r(o)|^{\gamma}}\sum_{y\in B_r(x)} f(y)&=\frac{1}{|B_r(o)|^{\gamma}}\sum_{n=0}^r\sum_{d(x,y)=n}f(y)\\&= \sum_{n=0}^\infty a_{r,\gamma}(n)\sum_{d(x,y)=n}f(y)=f\ast a_{r,\gamma}(x),
\end{align*}
where the convolution of $f$ with a radial function is defined in \cite[Formula (2.5)]{CMS98}. Hence,
\begin{equation}\label{convolution bound}
\mathcal M^{\gamma}f(x)=\sup_{r \in \mathbb N}|f|\ast a_{r,\gamma}(x)\lesssim |f|\ast a_\gamma(x)=:\mathcal{A}^\gamma |f|(x),\qquad  x\in T,
\end{equation}
 where $a_\gamma(x)=k^{-\gamma \Vert x\Vert }$.
\end{oss}

The two remarks alone are sufficient to give a full picture of the boundedness properties of $\mathcal{M}^{\gamma}$ when $\gamma>1$. Indeed, for $\gamma>1$ the kernel $a_\gamma\in L^p$ for every $p\in [1,\infty]$, and a simple application of Young's inequality gives us that the convolution operator $\mathcal{A}^\gamma$ is of strong type $(p,q)$ if $p\leq q$. Hence, it follows by the two remarks that, for $\gamma>1$, $\mathcal{M}^\gamma$ is of strong type $(p,q)$ if and only if $p\leq q$.
 
 The case $\gamma=1$ is much less trivial, but by now well understood. In this case, the $L^1$ to $L^1$ boundedness fails, since $\mathcal{M}\delta_o$ does not belong to $L^1$. Nevertheless, it was proved in \cite{CMS10, NT} that $\mathcal{M}$ is of weak type $(1,1)$. Since $\mathcal{M}$ is trivially bounded from $L^\infty$ to itself, it follows by interpolation, discrete $L^p$ spaces inclusions, and by Remark \ref{remark 1} that it is of strong type $(p,q)$ if and only if $1\leq p\leq q\neq 1$. 
 
The rest of the paper is devoted to the study of boundedness properties of fractional Hardy--Littlewood maximal operators $\mathcal{M}^\gamma$ for the remaining values $\gamma\in (0,1)$.

\section{Boundedness properties of $\mathcal{M}^\gamma$}\label{sec: boundedness}
%In this section, we determine values of $p,q\in [1,\infty]$ for which the fractional maximal operator $\mathcal{M}^\gamma$ is of restricted weak and strong type $(p,q)$. The optimality of these values is discussed in the next section.

The first two theorems we present should be considered as analogues of the endpoint results in \cite{CMS10,NT,V} for values of $\gamma$ not necessarily equal to $1$ or $1/2$.

%%%%CASO \gamma=1, ora compreso nei teoremi

% The case $\gamma=1$ is well understood. It was proved independently by Cowling, Meda, and Setti \cite{CMS10} and by Naor and Tao \cite{NT} that $\mathcal{M}$ is of weak type $(1,1)$. Since $\mathcal{M}$ is trivially bounded from $L^\infty$ to itself, it follows by interpolation and discrete $L^p$ spaces inclusions that it is of strong type $(t,s)$ whenever $s\ge t$ and $(t,s) \ne (1,1)$.  Conversely, we know that if $t>s$, $\mathcal{M}$ is unbounded from $L^t$ to $L^s$. Finally, $\mathcal{M}$ is not of strong type $(1,1)$ since $\mathcal M \delta_0$ does not belong to $L^1$. This gives a complete picture of the optimal range of exponents $(t,s)$ for which $\mathcal{M}$ is of strong type $(t,s).$
 
%Our scope is to extend the analysis of the strong $(t,s)$ boundedness of $\mathcal{M}^\gamma$ to the cases $\gamma\in (0,1)$.

\begin{thm}\label{complex interpolation}
  The operator $\mathcal{M}^{\gamma}$ is of restricted weak type $(1/(1-\gamma),1/\gamma)$, if $\gamma\in (0,1/2]$, and is bounded from $L^{1/(1-\gamma), 1/[2(1-\gamma)]}$ to $L^{1/(1-\gamma), \infty}$, if $\gamma\in (1/2,1]$.
\end{thm}

\begin{proof}
The case $\gamma=1$ is trivial, while the case $\gamma=1/2$ was proved in \cite[Theorem 5.1]{V}. Assume now $\gamma\neq 1,1/2$. 
%For $\gamma=1/2$, (i) and (ii) boil down to Veca's result.
For any couple of functions $\phi: T \to (0,\infty)$, $\psi: T \to \{z \in \mathbb C \ : \ |z|=1\}$, and any $z \in \mathbb C$, we define a linear operator $\mathcal{T}_{z,\phi,\psi}$ which acts on  $f\in \mathbb C^T$ as
\begin{align*}
    \mathcal{T}_{z,\phi,\psi}f(x)= \frac{1}{|B_{\phi(x)}(o)|^{z/2}} \sum_{y \in B_{\phi(x)}(x)} f(y)\psi(y).
\end{align*}
It is easy to see that for any $z\in \mathbb{C}$ and $x\in T$,
\begin{align}\label{supmax}
     \sup_{\phi,\psi}|\mathcal{T}_{z,\phi,\psi}f(x)|=\mathcal{M}^{\mathrm{Re}z/2}f(x),
\end{align}  where the supremum is taken over all functions $\phi, \psi$ as above. Hence, if $\mathcal{T}_{2\gamma,\phi,\psi}$ %satisfies the theorem, with operator norm which does not depend on $\phi$ and $\psi$. Given the claim, the result easily follows. Indeed, if $\mathcal{T}_{2\gamma,\phi,\psi}$ 
is bounded 
from $L^{p,s}$ to $L^{q,t}$  with operator norm which does not depend on $\phi$ and $\psi$, then for every $\varepsilon>0$ we may find functions $\phi$ and $\psi$, and a constant $C$ not depending on them, such that
\begin{equation*}
    \|\mathcal{M}^\gamma f\|_{q,t} -\varepsilon\leq \|\mathcal{T}_{2\gamma,\phi,\psi} f\|_{q,t} \leq C \|f\|_{p,s},
\end{equation*}
and by letting $\varepsilon \to 0^+$, we obtain that also $\mathcal{M}^\gamma$ is bounded from $L^{p,s}$ to $L^{q,t}$. Thus, it is enough to prove that  $\mathcal{T}_{2\gamma,\phi,\psi}$ is of restricted weak type $(1/(1-\gamma),1/\gamma)$, if $\gamma\in (0,1/2]$, and is bounded from $L^{1/(1-\gamma), 1/[2(1-\gamma)]}$ to $L^{1/(1-\gamma), \infty}$, if $\gamma\in (1/2,1)$, with operator norm that does not depend on $\phi$ and $\psi$.

To prove this, fix two functions $\phi: T \to (0,\infty)$, $\psi: T \to \{z \in \mathbb C \ : \ |z|=1\}$, set $\mathcal{T}_z=\mathcal{T}_{z,\phi,\psi}$ and $\Tilde{\mathcal{T}}_{z}=\mathcal{T}_{z+1}$, and consider the families of linear operators $\{\mathcal{T}_z\}_{z \in \overline{S}}$ and $\{\Tilde{\mathcal{T}}_{z}\}_{z\in \overline{S}}$, where $\overline{S}$ denotes the closure of the strip $S=\{0<\Re{z}< 1\}$. We aim to apply Cwikel and Janson's complex interpolation result \cite[Theorem 2]{CJ} to these families of operators and, respectively, to the spaces, $A_0=L^1$, $A_1=L^{2,1}$,  $B_0= L^{\infty}$, $B_1=L^{2,\infty}$ and $\Tilde{A_0}=L^{2,1}, \Tilde{A_1}=L^\infty, \Tilde{B_0}=L^{2,\infty}, \Tilde{B_1}=L^\infty$.

Observe that $A:=A_0\cap A_1=A_0$ and $\Tilde{A}:=\Tilde{A_0}\cap\Tilde{A_1}=\Tilde{A_0}$. We also set $B^+=L^1$.
It is easy to see that $\|\mathcal{M}^0 f\|_{\infty} = \|f\|_{1}$ for every $f \in L^1$. Moreover we know that $\mathcal{M}$ is of strong type $(\infty,\infty)$ and, by \cite[Theorem 5.1]{V}, that $\mathcal{M}^{1/2}$ is of restricted weak type $(2,2)$. By means of \eqref{supmax}, it follows that $\mathcal{T}_z$ is bounded from $A_0$ to $B_0$ when $\mathrm{Re}z=0$, and from $A_1$ to $B_1$ when $\mathrm{Re}z=1$, and $\Tilde{\mathcal{T}}_{z}$ is bounded from $\Tilde{A}_0$ to $\Tilde{B}_0$ when $\Re{z=0}$ and from $\Tilde{A}_1$ to $\Tilde{B}_1$ when $\Re{z=1}$.

We now observe that for every $b^+ \in B^+, a\in A, \Tilde{a} \in \Tilde{A}$, the two functions
\begin{align*}
    z \mapsto\langle b^+, \mathcal{T}_z a \rangle, \quad \text{and} \quad  z \mapsto\langle b^+, \Tilde{\mathcal{T}}_z \Tilde{a} \rangle,
\end{align*}
both belong to $H^\infty(\{\Re{z}\geq 0\})$, the space of bounded analytic functions on $\{\Re{z}> 0\}$ which are continuous on $\{\Re{z}\geq 0\}$. Indeed, a straightforward application of Morera's Theorem shows that both functions are entire. Moreover, for $\Re{z}\geq 0$ we have
\begin{align*}
  | \langle b^+, \mathcal{T}_z a \rangle |\le \sum_{x \in T} |b^+(x)|\mathcal{M}^0a(x) = \|b^+\|_1\|a\|_1,
\end{align*}
and
\begin{equation*}
    \begin{split}
        | \langle b^+, \Tilde{\mathcal{T}}_{z} \Tilde{a} \rangle |&\le \sum_{x \in T} |b^+(x)|\mathcal{M}^{1/2}\Tilde{a}(x) \leq \Vert \mathcal{M}^{1/2}\Tilde{a}\Vert_{2,\infty} \Vert b^+\Vert_{2,1}\leq \Vert \mathcal{M}^{1/2}\Vert_{L^{2,1}\to L^{2,\infty}}\Vert\Tilde{a}\Vert_{2,1} \Vert b^+\Vert_{1},
    \end{split}
\end{equation*}
where in the last two steps we used Hölder's inequality, Veca's theorem and the continuous inclusion of $L^{2,1}$ into $L^1$. It follows, in particular, that $z \mapsto\langle b^+, \mathcal{T}_z a \rangle$ and $z \mapsto\langle b^+, \Tilde{\mathcal{T}}_z \Tilde{a} \rangle$ belong to $H^\infty(\overline{S})$. Therefore, \cite[Theorem 2]{CJ} applies to both the families  $\{\mathcal{T}_{z}\}_{z\in \overline{S}}$ and $\{\Tilde{\mathcal{T}}_{z}\}_{z\in \overline{S}}$, and gives that for every $\theta \in (0,1)$
\begin{align*}
    \mathcal{T}_\theta: [A_0,A_1]_\theta \to [B_0,B_1]^\theta, \quad      \Tilde{\mathcal{T}}_\theta: [\Tilde{A}_0,\Tilde{A}_1]_\theta \to [\Tilde{B}_0, \Tilde{B}_1]^\theta,
\end{align*}
or equivalently (see for instance \cite{BL})
\begin{align*}
\mathcal{T}_\theta: L^{2/(2-\theta),1}\to L^{2/\theta,\infty}, \quad
    \Tilde{\mathcal{T}}_\theta: L^{2/(1-\theta),1/(1-\theta)} \to L^{2/(1-\theta),\infty},
\end{align*}
with operator norm not depending on $\phi,\psi$. Recalling that $\Tilde{\mathcal{T}}_\theta=\mathcal{T}_{\theta+1}$, and setting $\gamma=\theta/2$ then one has
\begin{equation*}
\mathcal{T}_{2\gamma}:
    \begin{cases}
    L^{1/(1-\gamma),1}\to L^{1/\gamma,\infty}, &\gamma\in (0,1/2)\\
    L^{1/(1-\gamma),1/[2(1-\gamma)]} \to L^{1/(1-\gamma),\infty}, &\gamma\in (1/2,1),
    \end{cases}
\end{equation*}
with operator norm not depending on $\phi,\psi$. This completes the proof.
\end{proof}

Next theorem gives us a supplementary endpoint result which is not possible to obtain by means of the above complex interpolation argument. To prove it we exploit a sharpened version of the Kunze-Stein phenomenon for Lorentz spaces on the homogeneous tree \cite[Theorem 1]{CMS98}.

\begin{thm}\label{p: homtree}
If $\gamma\in [1/2,1]$ the maximal operator $\mathcal M^{\gamma}$ is of restricted weak type $(1/\gamma,1/\gamma)$.  
\end{thm}
\begin{proof}
The case $\gamma=1$ follows by \cite{CMS10, NT}, and the case $\gamma=1/2$ was proved in \cite[Theorem 5.1]{V}. Assume now $\gamma\in (1/2,1)$. By Remark \ref{remark 2}, it is enough to prove the result for the linear convolution operator 
$$
\mathcal A^{\gamma}f(x)= f\ast a_\gamma(x),\qquad  x\in T.
$$ 
By Lemma \ref{lem: P}, $\|a_\gamma\|_{L^{1/\gamma,\infty}} \approx \sup_{n \in \mathbb N} k^{-\gamma n}k^{\gamma n}=1$.
% It is easy to check that $a_\gamma\in L^{1/\gamma,\infty}$. Indeed, for every $\lambda>0$, the condition $a_\gamma(x)>\lambda$ is equivalent to $\Vert x\Vert < \log_k (1/\lambda^{1/\gamma})$, thus 
%\begin{align*}
 %   |\{x \in T \ : \  |a_\gamma(x)|>\lambda\}| < |B_{\log_k (1/\lambda^{1/\gamma})}(o)| \approx k^{\log_k (1/\lambda^{1/\gamma})} = \frac{1}{\lambda^{1/\gamma}}.
%\end{align*}
Hence, by \cite[Theorem 1]{CMS98} we deduce that 
$$
\|\mathcal A^{\gamma}f\|_{{1/\gamma,\infty}}\lesssim \|f\|_{{1/\gamma,1}} \|a_\gamma\|_{{1/\gamma,\infty}}\lesssim \|f\|_{{1/\gamma,1}},
$$
i.e., that $\mathcal{A}^\gamma$ is of restricted weak type $(1/\gamma,1/\gamma)$.
\end{proof}

\begin{oss}
Observe that from Theorem \ref{p: homtree} one can easily deduce, arguing by duality, that $\mathcal M^\gamma$ is of restricted weak type $(1/(1-\gamma),1/(1-\gamma))$ for $\gamma\in (1/2,1)$. Indeed, it is easily seen that the linear operator $\mathcal A^\gamma$ is self-adjoint. Hence, by Hölder's inequality for Lorentz spaces,
\begin{equation*}
    |\langle \mathcal A^\gamma,g \rangle| = |\langle  f, \mathcal A^\gamma g \rangle|\leq \Vert f\Vert_{1/(1-\gamma),1} \Vert \mathcal A^\gamma g\Vert_{1/\gamma,\infty}\lesssim \Vert f\Vert_{1/(1-\gamma),1} \Vert  g\Vert_{1/\gamma,1}.
\end{equation*}
Passing to the supremum over all functions $g$ with $\Vert g\Vert_{1/\gamma,1}\leq 1$, we obtain
\begin{equation*}
   \Vert \mathcal A^\gamma f\Vert_{1/(1-\gamma),\infty}\lesssim \Vert f\Vert_{1/(1-\gamma),1}.
\end{equation*}
While this endpoint result would be sufficient for interpolation purposes, we remark that it is weaker than the result for $\gamma\in (1/2,1)$ obtained in Theorem \ref{complex interpolation}, since $L^{1/(1-\gamma),1}\not\subset L^{1/(1-\gamma),1/[2(1-\gamma)]}$ when $\gamma>1/2$.
\end{oss}

Thanks to the endpoint results obtained so far we can deduce strong type estimates.

%It is enough to recall the following easy embedding property for Lorentz spaces.
\begin{comment}
\begin{lem}\label{lem:inclusions}
If  an operator is of restricted weak type $(p_0,q_0)$ for some $p_0,q_0 \in [1,\infty)$, then it is of restricted weak type (strong type) $(p,q)$ for every $1 \le p\le p_0$ and $q_0\le q\le \infty$ (respectively, for every $1 \le p< p_0$ and $q_0<q\le \infty$).  
\end{lem}

\begin{proof}
For what concerns the restricted weak type boundedness it suffices to recall that for any $p,p_0 \in [1,\infty)$ and $s\in[1,\infty]$, the continuous inclusion $L^{p,s} \hookrightarrow L^{p_0,s}$ holds if $p \le p_0$. The statement regarding the strong type boundedness then follows from the general Marcinkiewicz interpolation theorem  \cite[Theorem 5.3.2]{BL}.
\end{proof}
\end{comment}

\begin{thm}\label{strong large gamma}
Let $\gamma\in (0,1]$. Then, $\mathcal M^{\gamma}$ is of strong type $(p,q)$ if one of the following conditions hold:
\begin{itemize}
    \item[(i)] $1\leq p\leq q\leq\infty$ and $q>1/\gamma$ and $p<1/(1-\gamma)$.
    \item[(ii)] $p=1/(1-\gamma)$ and $q=\infty$.
\end{itemize}
\end{thm}
\begin{proof}
For $\gamma=1$ the result boils down to the $(p,q)$ strong boundedness of $\mathcal{M}$ for $1\leq p\leq q\neq 1$ discussed in Section \ref{sec:preliminaries}. %From now on, let $\gamma\in (0,1)$.

If $\gamma \in (0,1/2]$ it follows from Theorem \ref{complex interpolation} and Lemma \ref{lem:inclusions} that $\mathcal M^{\gamma}$ is of strong type $(p,q)$ when $p< 1/(1-\gamma)$ and $q>1/\gamma$.
On the other hand, if $\gamma\in (1/2,1)$, combining Theorem \ref{complex interpolation} and Theorem \ref{p: homtree}, by interpolation we have that $\mathcal M^{\gamma}$ is of strong type $(t,t)$ (and hence of strong type $(p,q)$ with $p\leq t\leq q$), for $1/\gamma<t<1/(1-\gamma)$. It follows that $\mathcal M^{\gamma}$ is of strong type $(p,q)$ whenever $1\leq p\leq q\leq\infty$ with $p< 1/(1-\gamma)$ and $q>1/\gamma$. This proves (i).

To prove (ii), observe that for every $x \in T$ and $r\in \mathbb N,$ by  H\"older's inequality with exponents $p=1/(1-\gamma)$ and $p'=1/\gamma$, we obtain
\begin{align*}
    \frac{1}{|B_r(o)|^\gamma} \sum_{y \in B_r(x)} |f(y)| \le \|f\|_{1/(1-\gamma)}.
\end{align*}
Passing to the supremum on $r\in \mathbb{N}$, we get that $\mathcal{M}^\gamma f(x) \le \|f\|_{1/(1-\gamma)}$ for every $x \in T$, hence $\mathcal{M}^{\gamma}$ is of strong type $(1/(1-\gamma),\infty)$. 
\end{proof}

Let us remark that (i) above, which we deduce as a straightforward consequence of our endpoint results, for values of $\gamma\in[1/2,1]$ also follows from a more general theorem by Cowling, Meda and Setti \cite[Theorem 3.4]{CMS00}.

\section{Optimality results}\label{sec:opt}

In this section we discuss the optimality of the results obtained in Section \ref{sec: boundedness} under different points of view. Recall that Theorems \ref{complex interpolation} and \ref{p: homtree} give values of $(p,q)$ such that $\mathcal{M}^\gamma$ is bounded from $L^{p,s}$ to $L^{q,t}$ when $s=1$ and $t=\infty$. Our first optimality result shows that these results fail if $t\neq \infty$, for any fixed $s\in[1,\infty]$. 

\begin{prop}\label{counterexample 2} Let $\gamma \in (0,1)$. If $t \in [1,\infty)$, then $\mathcal{M}^{\gamma}$ is unbounded from $L^{1/(1-\gamma),s}$ to $L^{1/(1-\gamma),t}$ and from $L^{p,s}$ to $L^{1/\gamma,t}$, for every $p \in [1,\infty)$ and $s\in[1,\infty]$.
\end{prop} 
\begin{proof} For every $n \in \mathbb N$ set $f_n=\chi_{B_n(o)}$. It is easy to see that 
\begin{align*}
    \|f_n\|_{p,s} = |B_n(o)|^{1/p}, \qquad  s \in [1,\infty].
\end{align*}
For any $x \in B_n(o)$,
\begin{align*}
    \mathcal{M}^\gamma f_n(x)= \sup_{r \in \mathbb N}\frac{|B_r(x) \cap B_n(o)|}{|B_r(o)|^\gamma}   \ge \frac{|B_{n-\Vert x\Vert}(x)|}{|B_{n-\Vert x\Vert}(o)|^\gamma}\approx k^{(n-\|x\|)(1-\gamma)}.
\end{align*}
Since $\varphi_n(x):= f_n(x)k^{(n-\|x\|)(1-\gamma)}$ is radial, we may define $g_n(j)=\varphi_n(j)k^{j/p}$ and by Lemma \ref{lem: P}, for any $t\in [1,\infty)$ we have 
\begin{align*}
    \| \mathcal{M}^\gamma f_n\|_{{1/(1-\gamma),t}} &\gtrsim\|g_n \Vert_{L^t(\mathbb{N})}= \bigg(\sum_{j=0}^n k^{t(n-j)(1-\gamma)}k^{jt(1-\gamma)}\bigg)^{1/t}= n^{1/t}k^{n(1-\gamma)}. 
\end{align*}
It follows that for $t\in [1,\infty)$
\begin{align*}
  \frac{\| \mathcal{M}^\gamma f_n\|_{{1/(1-\gamma),t}}}{\|f_n\|_{{1/(1-\gamma),s}}}\gtrsim n^{1/t} \longrightarrow \infty, \qquad\mathrm{as} \ n \to \infty,
\end{align*}
which proves the first assertion. To prove the second, observe that if we apply $\mathcal{M}^\gamma$ to a Dirac delta centered at $o$ we get
\begin{align*}
    \mathcal{M}^\gamma\delta_o(x)= \frac{1}{|B_{\Vert x\Vert }(o)|^{\gamma}}, \qquad x \in T.
\end{align*} It follows by Lemma \ref{lem: P} that
\begin{align*}
    \|\mathcal{M}^\gamma \delta_o\|_{1/\gamma,t} \approx \bigg(\sum_{j=0}^\infty \frac{k^{j\gamma t}}{|B_{j }(o)|^{\gamma t}} \bigg)^{1/t}\approx \bigg(\sum_{j=0}^\infty 1 \bigg)^{1/t} =\infty,
\end{align*}and this concludes the proof.
\end{proof}
Now, we discuss the optimality of the values of $(p,q)$ for which we have established strong and restricted weak type boundedness of $\mathcal{M}^\gamma$.
\begin{prop}\label{optaxis}
Let $\gamma\in (0,1)$. The fractional maximal operator $\mathcal{M}^\gamma$ is not of strong type $(p,q)$ if either $q<p$, or $p=q=1/(1-\gamma)$, or $q \leq 1/\gamma$, or $p>1/(1-\gamma)$.
\end{prop}
\begin{proof}
The first assertion is Remark \ref{remark 1}, while the second and the third follow directly from Proposition \ref{counterexample 2}. To prove the last assertion, leveraging the discrete $L^p$ spaces inclusions, it is enough to show that $\mathcal{M}^\gamma$ is unbounded from $L^p$ to $L^\infty$ when $p>1/(1-\gamma)$.

To prove this, consider the sequence of functions defined by $f_n=\chi_{B_n(o)}$, $n \in \mathbb N$. We know from the proof of Proposition \ref{counterexample 2} that $\|f_n\|_{p}=|B_n(o)|^{1/p}$ and that $    \mathcal{M}^\gamma f_n(o) \ge |B_n(o)|^{1-\gamma}$. Thus, if $p>1/(1-\gamma)$,
\begin{align*}
    \frac{\| \mathcal{M}^\gamma f_n\|_{\infty}}{\|f_n\|_{p}} \ge |B_n(o)|^{1-\gamma-1/p}\to \infty, \qquad\mathrm{as} \ n \to \infty,
\end{align*}
which implies (ii).
\end{proof}

At this point, we have a complete description of the values of $p,q$ for which $\mathcal{M}^\gamma$ is of restricted weak type $(p,q)$ and for which is not.

\begin{cor}
Let $\gamma\in (0,1)$. Then, $\mathcal M^{\gamma}$ is of restricted weak type $(p,q)$ if and only if $1\leq p\leq q\leq\infty$ and $q\geq 1/\gamma$ and $p\leq 1/(1-\gamma)$. %$\min\{q,p'\}\geq 1/\gamma$.
\end{cor}
\begin{proof}
 The if part follows from Theorems \ref{complex interpolation} and \ref{p: homtree}, interpolation and Lemma \ref{lem:inclusions}. The only if part must hold true since otherwise, by Lemma \ref{lem:inclusions}, Proposition \ref{optaxis} would be contradicted.
\end{proof}

\begin{oss}\label{critical segment}
The description of the strong type boundedness region for $\mathcal{M}^\gamma$ is not complete yet for $\gamma\in (0,1)$. Indeed, the strong $(1/(1-\gamma),q)$ boundedness for $0<1/q< \min\{1-\gamma,\gamma\}$ has not been established nor disproved.
%when $0<1/q< 1-\gamma$, for $\gamma \in [1/2,1]$, and when $0<1/q< \gamma$, for $\gamma\in (0,1/2)$ has not been established nor disproved.
This remains an open question. The best we can say at present is what we obtain interpolating the result in Theorem \ref{complex interpolation} with (ii) in Theorem \ref{strong large gamma}, i.e., that $\mathcal{M}^\gamma$ is bounded from $L^{1/(1-\gamma),t}$ to $L^{q,\infty}$, where $t=q/(1+q-\gamma q)$ for $\gamma\in(1/2,1)$ and $t=q'$ for $\gamma\in(0,1/2]$.
\end{oss}
%The description of the strong type boundedness region for $\mathcal{M}^\gamma$ is not complete yet for $\gamma\in (0,1)$. Indeed, the strong $(1/(1-\gamma),q)$ boundedness for $0<1/q< \min\{1-\gamma,\gamma\}$ has not been established nor disproved.
%when $0<1/q< 1-\gamma$, for $\gamma \in [1/2,1]$, and when $0<1/q< \gamma$, for $\gamma\in (0,1/2)$ has not been established nor disproved.
%This remains an open question. The best we can say at present is what we obtain interpolating the result in Theorem \ref{complex interpolation} with (ii) in Theorem \ref{strong large gamma}, i.e., that $\mathcal{M}^\gamma$ is bounded from $L^{1/(1-\gamma),t}$ to $L^{q,\infty}$, where $t=q/(1+q-\gamma q)$ for $\gamma\in(1/2,1)$ and $t=q'$ for $\gamma\in(0,1/2]$.
%Unfortunately, we are not able to establish a (positive or negative) strong type $(1/(1-\gamma),q)$ boundedness result for the remaining values of $q$.

Next picture summarizes all the information that we obtained on the strong type boundedness of $\mathcal{M}^\gamma.$
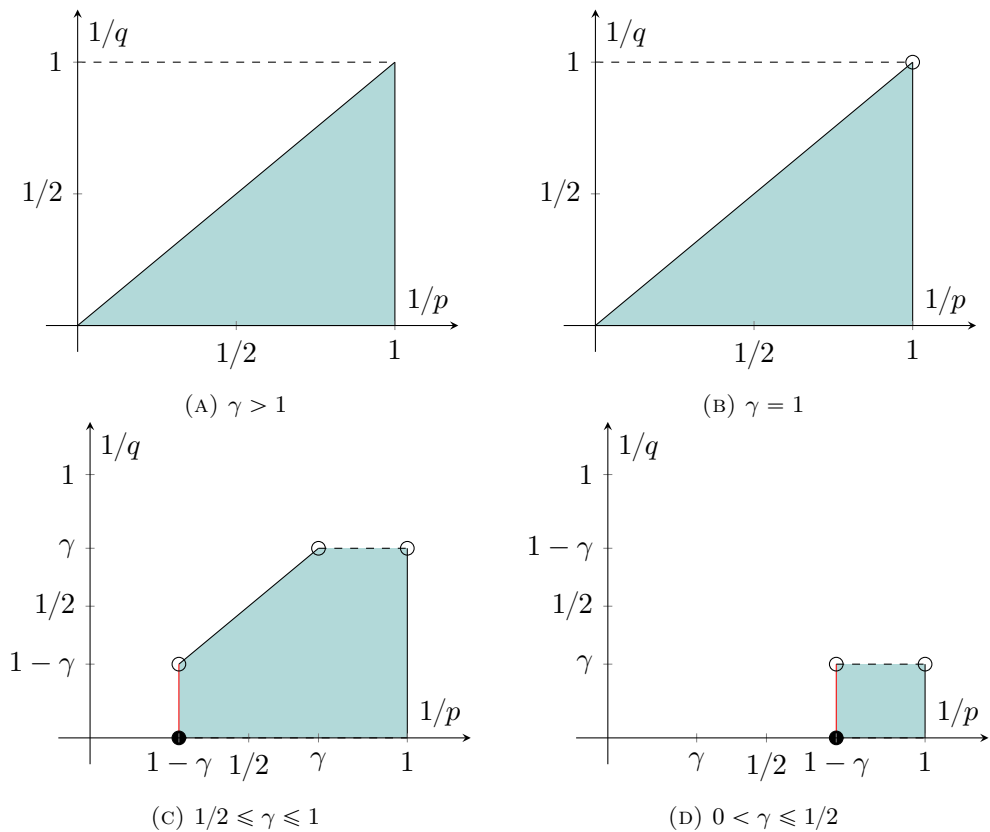
\begin{figure}[H]
\centering
\begin{tabular}{cc}
\subfloat[$\gamma> 1$]{
\begin{tikzpicture}[scale=0.8]%gamma=0.8, alpha=0.9
\begin{axis}[
xtick = {0, 0.5,1},
        xticklabels = {$0$,$1/2$,$1$},
ytick = {0, 0.5,1},
        yticklabels = {$0$,$1/2$,$1$},
axis lines=middle,
xlabel = {$1/p$},
ylabel = {$1/q$},
xmin=-0.1,
xmax=1.2,
ymin=-0.1,
ymax=1.2,
axis on top,
]
\addplot[name path=line0, domain = 0:1] {0};
\addplot[name path=line1, domain = 0:1] {x};
\addplot[name path=line2, dashed, domain = 0:1]{1}; %{((1-0.8*0.9)/(1-0.8*0.9-0.2))*(x-0.2)};
%\addplot[name path=line3, domain = 0.6:1] {(0.8*0.9/(0.8*0.9-1))*(x-1)};

\draw  (axis cs:{1,0}) -- (axis cs:{1,1});

%\draw [dashed] (axis cs:{0.8*0.9,0}) -- (axis cs:{0.8*0.9,0.8*0.9});
%\addplot[dashed, , domain = 0:1-0.8*0.9] {1-0.8*0.9};
%\addplot[dashed, , domain = 0:0.8*0.9] {0.8*0.9};
%\path (1, 1) node [shape=circle,draw,scale=0.5] {};
\addplot [teal!30] fill between [of = line0 and line1];
%\addplot [teal!30] fill between [of = line0 and line1, soft clip={domain=1-0.8*0.9:0.8*0.9}];
%\addplot [teal!30] fill between [of = line0 and line3, soft clip={domain=0.8*0.9:1}];
\end{axis}
\end{tikzpicture}
}
&\subfloat[$\gamma=1$]{
\begin{tikzpicture}[scale=0.8]%gamma=0.8, alpha=0.9
\begin{axis}[
xtick = {0, 0.5,1},
        xticklabels = {$0$,$1/2$,$1$},
ytick = {0, 0.5,1},
        yticklabels = {$0$,$1/2$,$1$},
axis lines=middle,
xlabel = {$1/p$},
ylabel = {$1/q$},
xmin=-0.1,
xmax=1.2,
ymin=-0.1,
ymax=1.2,
axis on top,
]
\addplot[name path=line0, domain = 0:1] {0};
\addplot[name path=line1, domain = 0:1] {x};
\addplot[name path=line2, dashed, domain = 0:1]{1}; %{((1-0.8*0.9)/(1-0.8*0.9-0.2))*(x-0.2)};
%\addplot[name path=line3, domain = 0.6:1] {(0.8*0.9/(0.8*0.9-1))*(x-1)};

\draw  (axis cs:{1,0}) -- (axis cs:{1,1});

%\draw [dashed] (axis cs:{0.8*0.9,0}) -- (axis cs:{0.8*0.9,0.8*0.9});
%\addplot[dashed, , domain = 0:1-0.8*0.9] {1-0.8*0.9};
%\addplot[dashed, , domain = 0:0.8*0.9] {0.8*0.9};
\path (1, 1) node [shape=circle,draw,scale=0.5] {};
\addplot [teal!30] fill between [of = line0 and line1];
%\addplot [teal!30] fill between [of = line0 and line1, soft clip={domain=1-0.8*0.9:0.8*0.9}];
%\addplot [teal!30] fill between [of = line0 and line3, soft clip={domain=0.8*0.9:1}];
\end{axis}
\end{tikzpicture}
}\\
\subfloat[$1/2\leq\gamma\leq 1$]{
\begin{tikzpicture}[scale=0.8]%gamma=0.8, alpha=0.9
\begin{axis}[
xtick = {1-0.8*0.9, 0.5,0.8*0.9,1},
        xticklabels = {$1-\gamma$,$1/2$,$\gamma$,$1$},
ytick = {1-0.8*0.9,0.5,0.8*0.9,1},
        yticklabels = {$1-\gamma$,$1/2$,$\gamma$,$1$},
axis lines=middle,
xlabel = {$1/p$},
ylabel = {$1/q$},
xmin=-0.1,
xmax=1.2,
ymin=-0.1,
ymax=1.2,
axis on top,
]
\addplot[name path=line0,dashed, domain = 1-0.8*0.9:1] {0};
\addplot[name path=line1,, domain = (1-0.8*0.9:0.8*0.9] {x};
%\addplot[name path=line2, domain = 0:0.48] {((1-0.8*0.9)/(1-0.8*0.9-0.2))*(x-0.2)};
%\addplot[name path=line3, domain = 0.6:1] {(0.8*0.9/(0.8*0.9-1))*(x-1)};
%\addplot[name path=line4, domain=0.6:1]{0.6}
\draw [name path=line2, red] (axis cs:{1-0.8*0.9,0}) -- (axis cs:{1-0.8*0.9,1-0.8*0.9});
%\draw [dashed] (axis cs:{0.8*0.9,0}) -- (axis cs:{0.8*0.9,0.8*0.9});
%\addplot[dashed, name path=line2, domain = 0:1-0.8*0.9] {1-0.8*0.9};
\addplot[name path=line3, dashed, ,  domain = 0.8*0.9:1] {0.8*0.9};
\draw  (axis cs:{1,0}) -- (axis cs:{1,0.8*0.9});
\addplot [teal!30] fill between [of = line0 and line2, soft clip={domain=1-0.8:1-0.8*0.9}];
\path (0.8*0.9, 0.8*0.9) node [shape=circle,draw,scale=0.5] {};
\path (1, 0.8*0.9) node [shape=circle,draw,scale=0.5] {};
\path (1-0.8*0.9, 1-0.8*0.9) node [shape=circle,draw,scale=0.5] {};
\path (1-0.8*0.9, 0) node [shape=circle, fill=black, draw,scale=0.5] {};
\addplot [teal!30] fill between [of = line0 and line1, ];
\addplot [teal!30] fill between [of = line0 and line3, soft clip={domain=0.8*0.9:1}];
\end{axis}
\end{tikzpicture}
}
&\subfloat[$0<\gamma\leq 1/2$]{
\begin{tikzpicture}[scale=0.8]%gamma=0.8, alpha=0.9
\begin{axis}[
xtick = {1-0.8*0.9, 0.5,0.8*0.9,1},
        xticklabels = {$\gamma$,$1/2$,$1-\gamma$,$1$},
ytick = {1-0.8*0.9,0.5,0.8*0.9,1},
        yticklabels = {$\gamma$,$1/2$,$1-\gamma$,$1$},
axis lines=middle,
xlabel = {$1/p$},
ylabel = {$1/q$},
xmin=-0.1,
xmax=1.2,
ymin=-0.1,
ymax=1.2,
axis on top,
]
\addplot[name path=line0,dashed, domain = 0.8*0.9:1] {0};
%\addplot[name path=line1, dashed, domain = 0:1] {x};
%\addplot[name path=line2, domain = 0:0.48] {((1-0.8*0.9)/(1-0.8*0.9-0.2))*(x-0.2)};
%\addplot[name path=line3, domain = 0.6:1] {(0.8*0.9/(0.8*0.9-1))*(x-1)};
%\addplot[name path=line4, domain=0.6:1]{0.6}
\draw [name path=line2, red] (axis cs:{0.8*0.9,0}) -- (axis cs:{0.8*0.9,1-0.8*0.9});
%\draw [dashed] (axis cs:{0.8*0.9,0}) -- (axis cs:{0.8*0.9,0.8*0.9});
%\addplot[dashed, name path=line2, domain = 0:1-0.8*0.9] {1-0.8*0.9};
\addplot[name path=line3, dashed, ,  domain = 0.8*0.9:1] {1-0.8*0.9};
\draw  (axis cs:{1,0}) -- (axis cs:{1,1-0.8*0.9});
\addplot [teal!30] fill between [of = line0 and line2, soft clip={domain=1-0.8:1-0.8*0.9}];
\path (0.8*0.9, 1-0.8*0.9) node [shape=circle,draw,scale=0.5] {};
\path (1, 1-0.8*0.9) node [shape=circle,draw,scale=0.5] {};
\path (0.8*0.9, 0) node [shape=circle, fill=black, draw,scale=0.5] {};
%\addplot [teal!30] fill between [of = line0 and line1, ];
\addplot [teal!30] fill between [of = line0 and line3, soft clip={domain=0.8*0.9:1}];
\end{axis}
\end{tikzpicture}
}
\end{tabular}

\caption{$\mathcal{M}^\gamma$ is of strong type $(p,q)$ if the point $(1/p,1/q)$ lies in the colored region or along the portion of its boundary made of black continuous lines and the black point. It is unbounded outside the colored region, at circled points and along the dashed line. We don't know whether it is bounded or not along the red segment.}\label{fig:2}
\end{figure}

We end the section with a last natural question concerning optimality, that is, if the endpoint results obtained in Theorems \ref{complex interpolation} and \ref{p: homtree} can be improved to $L^{p,s}$ to $L^{q,\infty}$ boundedness results for some $s>1$. In the particular case $\gamma=1/2$, the two theorems reduce to Veca's result \cite{V}, i.e., that $\mathcal{M}^{1/2}$ is of restricted weak type $(2,2)$, and in this case we are able to show that the answer to the above question is negative. This represents a discrete counterpart of the result by Ionescu \cite{Ionescu2} obtained in the setting of non-compact symmetric spaces mentioned in the introduction.

In order to prove this result, as well as the one that follows, it is useful to highlight the following formula, which holds true for any radial nonnegative function $f\in\mathbb{C}^T$ and  can be obtained by a straightforward computation,
\begin{equation}\label{rad1}
 \sum_{y \in S_n(x)} f(y) \approx \begin{cases} \sum_{j=0}^nf(\Vert x\Vert +n-2j)k^{n-j} &\text{if $n \le \Vert x\Vert $,} \\ 
   \sum_{j=0}^{\Vert x\Vert }f(\Vert x\Vert +n-2j)k^{n-j} &\text{otherwise.}
    \end{cases}
\end{equation}

\begin{thm}\label{optM12} For every $s>1$,
$\mathcal{M}^{1/2}$ is unbounded from $L^{2,s}$ to $L^{2,\infty}.$
\end{thm}
\begin{proof}
Fix $s>1$ and $1/s<\beta<1$. Define $g \in \mathbb C ^{T}$ by
\begin{align*}
    g(x)=\frac{k^{-\Vert x\Vert /2}}{(1+\Vert x\Vert )^\beta}, \qquad x \in T.
\end{align*}
Since $g$ is radial, by Lemma \ref{lem: P}, $\|g\|_{{2,s}}^s \approx  \sum_{n=0}^\infty (1+n)^{-\beta s}$, thus $g$ belongs to $L^{2,s}$. Moreover by formula \eqref{rad1}, for any $x\in T$,
\begin{equation*}
    \sum_{y \in S_{\Vert x\Vert }(x)}g(y)\approx \sum_{j=0}^{\Vert x\Vert } \frac{k^{-2(\Vert x\Vert -j)/2}k^{\Vert x\Vert -j}}{(1+2(\Vert x\Vert -j))^{\beta}}=\sum_{j=0}^{\Vert x\Vert } \frac{1}{(1+2(\Vert x\Vert -j))^{\beta}}\approx (1+\Vert x\Vert )^{1-\beta}.
\end{equation*}
It follows,
\begin{align*}
   \mathcal{M}^{1/2}g(x) \geq \frac{1}{|B_{\Vert x\Vert}(o)|^{1/2}}\sum_{y \in S_{\Vert x\Vert}(x)} g(y) \gtrsim k^{-\Vert x\Vert /2}(1+\Vert x\Vert )^{1-\beta}=:m(x).
\end{align*} 
Since $m$ is radial and $\beta<1$, again by Lemma \ref{lem: P},
\begin{align*}
    \|m\|_{{2,\infty}} \approx \|m(\cdot)k^{|\cdot|/2}\|_{L^\infty(\mathbb N)}=+\infty.
\end{align*}  Hence $\mathcal{M}^{1/2}$ does not map $L^{2,s}$ in $L^{2,\infty}$.
\end{proof}
\begin{comment}
\textcolor{red}{
\begin{thm}\label{optM12} For every $s>1$,
$\mathcal{M}^{1/2}$ is unbounded from $L^{p,s}$ to $L^{2,\infty}$, for every $p\in [1,\infty)$.
\end{thm}
\begin{proof}
By Lorentz spaces inclusions, it is enough to prove the result for $p=1$. Fix $s>1$ and $1/s<\beta<1$. Define $g \in \mathbb C ^{T}$ by
\begin{align*}
    g(x)=\frac{k^{-\Vert x\Vert}}{(1+\Vert x\Vert )^\beta}, \qquad x \in T.
\end{align*}
Since $g$ is radial, by Lemma \ref{lem: P}, $\|g\|_{{1,s}}^s \approx  \sum_{n=0}^\infty (1+n)^{-\beta s}$, thus $g$ belongs to $L^{1,s}$ . Moreover by formula \eqref{rad1}, for any $x\in T$,
\begin{equation*}
    \sum_{y \in S_{\Vert x\Vert }(x)}g(y)\approx \sum_{j=0}^{\Vert x\Vert } \frac{k^{-(\Vert x\Vert -j)}}{(1+2(\Vert x\Vert -j))^{\beta}}\geq\sum_{j=0}^{\Vert x\Vert } \frac{1}{(1+2(\Vert x\Vert -j))^{\beta}}\approx (1+\Vert x\Vert )^{1-\beta}.
\end{equation*}
It follows,
\begin{align*}
   \mathcal{M}^{1/2}g(x) \geq \frac{1}{|B_{\Vert x\Vert}(o)|^{1/2}}\sum_{y \in B_{\Vert x\Vert}(x)} g(y) \gtrsim k^{-\Vert x\Vert /2}(1+\Vert x\Vert )^{1-\beta}=:m(x).
\end{align*} 
Since $m$ is radial and $\beta<1$, again by Lemma \ref{lem: P},
\begin{align*}
    \|m\|_{{2,\infty}} \approx \|m(\cdot)k^{|\cdot|/2}\|_{L^\infty(\mathbb N)}=+\infty.
\end{align*}  Hence $\mathcal{M}^{1/2}$ does not map $L^{p,s}$ in $L^{2,\infty}$.
\end{proof}}
\end{comment}

More in general, one may ask whether a similar strategy can be applied also to values of $\gamma\ne 1/2$. %Next proposition makes it clear that this is not the case, by showing that no radial function can serve as a counterexample for the unboundedness of $\mathcal{M}^\gamma$ from $L^{p,s}$ to $L^{q,\infty}$ when $s>1$ and $(p,q)$ the values at which $\mathcal{M}^\gamma$ is of restricted weak type provided by Theorems \ref{complex interpolation} and \ref{p: homtree}.
Next proposition makes it clear that this is not the case, by showing that no radial function can serve as a counterexample to prove the optimality of Theorems \ref{complex interpolation} and \ref{p: homtree} with respect to the parameter $s$. 

%Indeed, it shows (in particular) that $\mathcal{M}^\gamma$ maps continuously $(L^{1/\gamma,s})^{\#}$ to $L^{1/\gamma,\infty}$ and $(L^{1/(1-\gamma),s})^{\#}$ to $L^{1/(1-\gamma),\infty}$ when $\gamma>1/2$, and $(L^{1/(1-\gamma),s})^{\#}$ to $L^{1/\gamma,\infty}$ when $\gamma<1/2$, for every $s \in [1,\infty]$, where $(L^{p,s})^\#$ denotes the space of radial functions in $L^{p,s}.$

%\begin{prop}\label{prop:radial}
%For every $s \in [1,\infty]$, $\mathcal{M}^\gamma$ maps continuously $(L^{p,s})^{\#}$ to $L^{p,\infty}$ if $1/\gamma\leq p\leq 1/(1-\gamma)$ and $p \neq 2$, and to $L^{p',\infty}$ if $p<2$ and $p\leq 1/(1-\gamma)$.
%\end{prop}

\begin{prop}\label{prop:radial}
Let $(L^{p,s})^\#$ denote the space of radial functions in $L^{p,s}.$ Then, $\mathcal{M}^\gamma$ maps continuously $(L^{1/\gamma,s})^{\#}$ to $L^{1/\gamma,\infty}$ and $(L^{1/(1-\gamma),s})^{\#}$ to $L^{1/(1-\gamma),\infty}$ when $\gamma>1/2$ and $(L^{1/(1-\gamma),s})^{\#}$ to $L^{1/\gamma,\infty}$ when $\gamma<1/2$, for every $s \in [1,\infty]$.
\end{prop}

\begin{proof}
 %The proof is based on the following formula which can be obtained by a straightforward computation
\begin{comment}\begin{equation}\label{rad1}
\begin{split}
 \sum_{y \in S_n(x)} f(y) &\approx \begin{cases} \sum_{j=0}^nf(\Vert x\Vert +n-2j)k^{n-j} &\text{if $n \le \Vert x\Vert $,} \\ 
   \sum_{j=0}^{\Vert x\Vert }f(\Vert x\Vert +n-2j)k^{n-j} &\text{otherwise.}
    \end{cases}
\end{split}
\end{equation}
\end{comment}
 Let $f$ be a nonnegative  function in $(L^{p,s})^{\#}$, so that by Lemma \ref{lem: P} $g(\cdot):= f(\cdot)k^{(\cdot)/p}$ belongs to $L^{s}(\mathbb N)$ and $\|g\|_{L^s(\mathbb N)} \approx \|f\|_{{p,s}}$. Rewriting \eqref{rad1} in terms of $g$ and then applying H\"older's inequality we get
\begin{equation}\label{rad2}
\begin{split}
 \frac{1}{k^{n\gamma}}\sum_{y \in S_n(x)} f(y) &\approx \begin{cases} k^{-n\gamma}\sum_{j=0}^n g(\Vert x\Vert +n-2j)k^{j(2/p-1)}k^{n/p'-\Vert x\Vert/p} &\text{if $n \le \Vert x\Vert $,} \\ 
   k^{-n\gamma}\sum_{j=0}^{\Vert x\Vert} g(\Vert x\Vert +n-2j)k^{j(2/p-1)}k^{n/p'-\Vert x\Vert/p} &\text{otherwise,}
    \end{cases}\\
    &\lesssim \begin{cases}
      k^{n(1/p'-\gamma)}k^{-\Vert x\Vert /p}\|g\|_{L^s(\mathbb N)} &\text{if $p> 2$,} \\
      k^{n(1/p'-\gamma)}k^{-\Vert x\Vert /p'}\|g\|_{L^s(\mathbb N)} &\text{if $p<2$,}\\
      k^{n(1/p-\gamma)}k^{-\Vert x\Vert /p}\|g\|_{L^s(\mathbb N)} &\text{if $p<2$ and $n\leq \Vert x\Vert$.}
       \\ 
      \end{cases}
    %&\lesssim \sum_{j=0}^{\Vert x\Vert}f(\Vert x\Vert +n-2j)k^{n-j}.
\end{split}
\end{equation}
It follows,
\begin{align*}
      \mathcal{M}^\gamma f(x) \lesssim 
      \begin{cases}
      k^{-\Vert x\Vert/p}\|g\|_{L^s(\mathbb N)} &\text{if $2< p\leq 1/(1-\gamma)$,} \\
      k^{-\Vert x\Vert/p'}\|g\|_{L^s(\mathbb N)} &\text{if $p<2$ and $p\leq 1/(1-\gamma)$,}\\
      k^{-\Vert x\Vert/p}\|g\|_{L^s(\mathbb N)} &\text{if $1/\gamma\leq p< 2$ and $n\leq \Vert x\Vert$,} \\ 
      k^{-\gamma\Vert x\Vert}\|g\|_{L^s(\mathbb N)} &\text{if $1/\gamma\leq p< 2$ and $n> \Vert x\Vert$,}
      \end{cases}
\end{align*}
and since $x\mapsto k^{-\Vert x\Vert/t}$ belongs to $L^{p,\infty}$ if and only if $t\leq p$ and $ \|g\|_{L^s(\mathbb N)}\approx \|f\|_{{p,s}}$, we have
\begin{align*}
      \mathcal{M}^\gamma: 
      \begin{cases}
      (L^{p,s})^{\#}\to L^{p,\infty} &\text{if $1/\gamma\leq p\leq 1/(1-\gamma)$ and $p \neq 2$}, \\ 
      (L^{p,s})^{\#}\to L^{p',\infty} &\text{if $p<2$ and $p\leq 1/(1-\gamma)$,} \end{cases}
\end{align*}
The result follows by choosing $p=1/(1-\gamma)$.
%In particular, $\mathcal{M}^\gamma$ maps continuously $(L^{1/\gamma,s})^{\#}\to L^{1/\gamma,\infty}$ and $(L^{1/(1-\gamma),s})^{\#}\to L^{1/(1-\gamma),\infty}$ when $\gamma>1/2$ and $(L^{1/(1-\gamma),s})^{\#}\to L^{1/\gamma,\infty}$ when $\gamma<1/2$, for every $s \in [1,\infty]$.
\end{proof}

%In the next Proposition we show that for $s=\frac{1}{2(1-\gamma)}$ the boundedness property i) in Proposition \ref{prop:radial} extends from radial functions to all functions in $L^{1/(1-\gamma),s}$, hence refining the restricted weak type $(1/(1-\gamma),1/(1-\gamma))$ boundedness of $\mathcal{M}^\gamma$ for $\gamma\in (1/2,1]$ obtained in Theorem \ref{p: homtree}. More precisely we prove the following.

%Proposition refinement (complex interpolation at 1/(1-\gamma) for gamma>1/2)

\section{Final remarks}
Let $\omega:T\to\mathbb{R}_+$ and $w(A)=\sum_{x\in T}\omega(x)$, $A\subseteq T$, the associated measure.
In a recent paper \cite{GR} Ghosh and Rela proved that $\mathcal{M}^\gamma$ is bounded from $L^p(\omega)$ to $L^q(\omega)$, with $1<p\leq q<\infty$ and $\gamma\in (0,1)$, if there exists $\varepsilon\in (0,1)$ such that $\omega \in \mathcal{
Z}^{\varepsilon, 1-\gamma}_{p,q}$, i.e., if
\begin{align*}
  \sum_{x \in E} \omega(F \cap S_r(x)) \leq C k^{\varepsilon r \gamma} \omega(E)^{1/p}\omega(F)^{1-1/q}, \quad \text{for some } C>0.
\end{align*}
It is natural to check if this sufficient condition can provide a positive answer to the only question we left open regarding strong boundedness (see Remark \ref{critical segment}). %that is, if $\mathcal{M^\gamma}$ is bounded (or not) from $L^{1/(1-\gamma)}$ to $L^q$, with $0<1/q<\min\{\gamma, 1-\gamma\}$.
We are showing here that this is not the case, since $|\cdot|\not\in \mathcal{Z}^{\varepsilon, 1-\gamma}_{1/(1-\gamma),q}$, for any $\varepsilon\in (0,1)$, and any $1<q<\infty$. To see this, observe that for any $\gamma,\gamma'\in(0,1)$, $\mathcal{
Z}^{\varepsilon, 1-\gamma}_{p,q}=\mathcal{
Z}^{\varepsilon', 1-\gamma'}_{p,q}$, where $\varepsilon'=\varepsilon\gamma/\gamma'$. Now, choose $\gamma'\in (\varepsilon\gamma,\gamma)$ and $p=1/(1-\gamma)$, so that $p>1/(1-\gamma')$. If $|\cdot|\in \mathcal{Z}^{\varepsilon, 1-\gamma}_{p,q}$, then we would have that $\mathcal{M}^{\gamma'}$ is bounded from $L^p$ to $L^q$, with $p>1/(1-\gamma')$, which contradicts Theorem \ref{optaxis}. This shows that the result in \cite{GR} cannot improve our Theorem \ref{strong large gamma}.

{\bf{Acknowledgments.}} The authors are indebted to Stefano Meda and Maria Vallarino for the many precious discussions on the topics of the paper.

\bibliographystyle{abbrv}
{\small
\bibliography{references2}}

\end{document}